\documentclass[11pt, a4paper]{amsart}
\usepackage{amssymb, array, amsmath, amscd, pdfpages, enumerate, amsthm, fixltx2e, setspace, hyperref,mathtools, overpic,epstopdf}

\usepackage[utf8]{inputenc}
\usepackage{amssymb}
\usepackage{bm}
\usepackage{graphicx} 

\newtheorem*{RORP}{The Robust Output Regulation Problem}

\newcommand{\indim}{r}

\newcommand{\yref}{y_{\mbox{\scriptsize\textit{ref}}}}
\newcommand{\yrefhat}{\hat{y}_{\mbox{\scriptsize\textit{ref}}}}

\newcommand{\Ops}{\mc{O}}

\newcommand{\eps}{\epsilon}

\renewcommand{\vec}[1]{#1}

\newcommand{\Hinf}{H_\infty}

\newcommand{\yrk}[1][k]{a_{#1}}

\usepackage{ifthen}
\newcommand*{\keyterm}[1]{\emph{#1}}
\DeclareMathOperator{\re}{Re}			
\DeclareMathOperator{\rank}{rank}		
\DeclareMathOperator{\Span}{span}		
\newcommand{\ran}{\mathcal{R}}						
\renewcommand{\ker}{\mathcal{N}}					

\newcommand*{\ldelim}[2]{\csname#1l\endcsname#2}   
\newcommand*{\rdelim}[2]{\csname#1r\endcsname#2}   
\newcommand*{\mdelim}[2]{\csname#1m\endcsname#2}   

\newcommand{\gd}{\delta}

\newcommand{\gw}{\omega}

\newcommand*{\C}{{\mathbb{C}}}     

\newcommand*{\norm}[1]{\lVert#1\rVert}

\newcommand*{\set} [1]{\{#1\}}
\newcommand*{\setm}[2]{\{\,#1\mid#2\,\}}   

\newcommand*{\Abs}[2][default]{\ifthenelse{\equal{#1}{default}}{\left\lvert#2\right\rvert}{\ldelim{#1}{\lvert}#2\rdelim{#1}{\rvert}}}
\newcommand*{\Norm}[2][default]{\ifthenelse{\equal{#1}{default}}{\left\lVert#2\right\rVert}{\ldelim{#1}{\lVert}#2\rdelim{#1}{\rVert}}}

\newcommand*{\Iprod}[3][default]{\ifthenelse{\equal{#1}{default}}{\left\langle#2,#3\right\rangle}{\ldelim{#1}{\langle}#2,#3\rdelim{#1}{\rangle}}}
\newcommand*{\Dualpair}[3][default]{\ifthenelse{\equal{#1}{default}}{\left\langle#2,#3\right\rangle}{\ldelim{#1}{\langle}#2,#3\rdelim{#1}{\rangle}}}

\newcommand*{\List}[2][1]{\set{#1,\ldots,#2}}

\newcommand*{\Set}[2][default]{\ifthenelse{\equal{#1}{default}}{\left\{#2\right\}}{\ldelim{#1}{\{}#2\rdelim{#1}{\}}}}

\newcommand*{\dda}[3][1]{\ifthenelse{\equal{#1}{1}}{\frac{d#3}{d#2}}{\frac{d^{#1}#3}{d#2^{#1}}}}
\newcommand*{\ddb}[2][1]{\ifthenelse{\equal{#1}{1}}{\frac{d}{d#2}}{\frac{d^{#1}}{d#2^{#1}}}}
\newcommand*{\pd}[3][1]{\ifthenelse{\equal{#1}{1}}{\frac{\partial{#2}}{\partial{#3}}}{\frac{\partial^{#1}{#2}}{\partial#3^{#1}}}}
\newcommand*{\pdb}[2][1]{\ifthenelse{\equal{#1}{1}}{\frac{\partial}{\partial{#2}}}{\frac{\partial^{#1}}{\partial#2^{#1}}}}

\newcommand*{\inv}{^{-1}}

\renewcommand{\vec}[1]{\underline{#1}}

\newcommand{\mc}[1]{\mathcal{#1}}

\newcommand{\eq}[1]{\begin{align*}#1\end{align*}}
\newcommand{\eqn}[1]{\begin{align}#1\end{align}}

\newtheorem{theorem}{Theorem}[section]
\newtheorem{lemma}[theorem]{Lemma}
\newtheorem{corollary}[theorem]{Corollary}
\theoremstyle{definition}

\newtheorem{remark}[theorem]{Remark}

\begin{document}

\title{Reduced Order Internal Models in the Frequency Domain}

\thispagestyle{plain}

\author{Petteri Laakkonen and Lassi Paunonen}
\address{$^\dagger$ Laboratory of Mathematics, Tampere University of Technology, PO.\ Box 553, 33101 Tampere, Finland}
\email{petteri.laakkonen@tut.fi, lassi.paunonen@tut.fi}

\begin{abstract}
The internal model principle states that all robustly regulating controllers must contain a suitably reduplicated internal model of the signal to be regulated. Using frequency domain methods, we show that the number of the copies may be reduced if the class of perturbations in the problem is restricted. We present a two stage design procedure for a simple controller containing a reduced order internal model achieving robust regulation. The results are illustrated with an example of a five tank laboratory process where a restricted class of perturbations arises naturally.
\end{abstract}


\maketitle

\section{Introduction}

The main goal in output regulation is to find
a controller such that the output of a given plant asymptotically follows a given reference signal generated by an exosystem. It is known that a regulating 
feedback
controller contains a built-in copy of the exosystem \cite{Fra77}. Robustness of regulation is needed in order to make the controller work despite some perturbations of the plant, e.g., parameter uncertainties and modelling errors. If the controller is required to tolerate arbitrary small perturbations of the plant, then the internal model principle due to Francis and Wonham \cite{FraWon75a} and Davison \cite{Dav76} states that if the plant has $p$-dimensional output space, then every robustly regulating controller must contain a $p$-fold copy or in short $p$-copy of the exosystem.


In this paper we study
the robust regulation problem in a situation 
where the controller is only required to 
tolerate uncertainties from a restricted class $\Ops$ of perturbations.
Such a situation can arise due to
several different reasons. 
In the simplest situation, $\Ops$ contains only a finite number of plants, for example, if the controller is required to function after specific component failures~\cite{LocSchi11}. 
In the case of only 
one possible failure, the original plant $P(\cdot)$ changes to a new plant $P_f(\cdot)$ and $\Ops=\{P(\cdot), P_f(\cdot)\}$. 
On the other hand, the class $\Ops$ becomes infinite in a situation where the values of some specific parameters of the plant are not known accurately
\cite{KimSkoMorBraatz2014,PauPoh13a}. Our example in Section~\ref{sec:example} illustrates the latter case.

In the situation where robustness is only required with respect to a given class $\Ops$ of perturbations, 
it is natural to ask if the controller must contain a full $p$-copy internal model of the exosystem.
This problem was studied by Paunonen in~\cite{PauPoh13a,Pau15a} using state space methods. It was shown in~\cite{PauPoh13a} that the 
$p$-copy internal model guaranteeing robustness with respect to
all small perturbations can be relaxed in 
many
situations, and this observation 
leads to design of controllers with so-called \keyterm{reduced order internal models}. 

In this paper, we introduce frequency domain conditions for 
a controller to achieve output regulation and robustness with respect to a given class $\Ops$ of perturbations. Our results give a precise meaning to reduced order internal models in the frequency domain. In addition, we present methods for
constructing controllers with reduced order internal models. 
Our constructions result in 
minimal complexity requirements for the number of copies built into the robustly regulating controllers.

The reference signals considered in this article are linear combinations of sinusoids, and in particular 
finite approximations of uniformly continuous periodic signals. In explicit,
we choose the reference signal to be of the form
\eqn{
\label{eq:yrefintro}
\yref(t) = \sum_{k=1}^q \yrk e^{i\gw_k t} 
}
with distinct fixed real numbers $\gw_k$ and $a_k\in\C^p\setminus\{0\}$.
Our aim is to characterize conditions for controllers that make the output of the systems to converge to the reference signal as $t\to\infty$ with all plants in a given class $\Ops$ of perturbed systems.



In the first part of this paper we present our theoretical results. Our first main result
is the frequency domain formulation of the internal model principle for reduced order internal models.
The result states that
a stabilizing controller
\eqn{
\label{eq:ControllerIntro}
C(s)=\sum_{k=1}^q\frac{C_k}{s-i\gw_k}+C_{0}(s),
}
where $C_0(\cdot)$ is analytic at $i\omega_k$, is robustly regulating for the class $\Ops$ of perturbed plants if and only if
\eqn{
\label{eq:Robcharintro}
\yrk \in \ran(\widetilde{P}(i\gw_k)C_k), \qquad \forall k\in \List{q}
}
for all $\widetilde{P}(\cdot)\in \Ops$. The component $(s-i\gw_k)^{-1} C_k$ is the internal model of the frequency component $\yrk e^{i\gw_k t}$ of the reference signal \eqref{eq:yrefintro} and the condition \eqref{eq:Robcharintro} shows how to align the pole of the controller with the corresponding frequency component.
This is the frequency domain analogue of the time domain condition presented in~\cite{PauPoh13a}.

The condition~\eqref{eq:Robcharintro} leads to our second main results that
gives a lower bound for the ranks of the matrices $C_k$ in the controller, i.e. it gives the size of the minimal internal model required for robust regulation. In particular, if the plants have 
$p$ 
inputs and outputs and $\widetilde{P}(\cdot)\in\Ops$ are invertible at $i\gw_k$, then the lower bounds for the ranks of $C_k$ in the controller~\eqref{eq:ControllerIntro} that is robust with respect to the class $\Ops$ are
\eqn{
\label{eq:Rankcondintro}
\hspace{-1ex}
\rank (C_k)\geq \sigma_k:= 
\dim\left( \Span\setm{\widetilde{P}(i\gw_k)\inv\yrk}{\widetilde{P}(\cdot)\in\Ops}\right) .
}
The controller constructions presented later in Section~\ref{sec:ContDesign} show that the lower bounds $\sigma_k$ are optimal in the sense that robustness with respect to a class $\Ops$ can be achieved with a controller
satisfying $\rank (C_k)=\sigma_k$ for $k\in \List{q}$.
In the frequency domain, a
controller containing a full $p$-copy internal model 
satisfies $\rank (C_k)=p$ for all $k$~\cite{LaaPoh15}. 
We therefore say that 
the controller~\eqref{eq:ControllerIntro} contains a \keyterm{reduced order internal model} of the reference signal if it satisfies the conditions for robustness for a class $\Ops$ of perturbations and $\rank (C_k)<p$ for some $k$.
%
In a situation where $\sigma_k<p$ for some $k$, e.g., when $p>2$ and $\Ops$ contains only two plants, robust output regulation can then be achieved without the full internal model of the reference signal.

In the second part of the paper we construct a 
controller that solves the robust output regulation problem for a given class $\Ops$ of perturbations. 
In the design procedure we first stabilize the system and then design a robustly regulating controller for the stabilized plant.  
%
The robust regulation of the stabilized plant is achieved using a controller
\eq{
C_r(s) = \eps \sum_{k=1}^q \frac{C_k}{s-i\gw_k},
}
where
$C_k\in \C^{m\times p}$ are
chosen in such a way that they satisfy the regulation condition \eqref{eq:Robcharintro} and have ranks $\sigma_k$ defined in~\eqref{eq:Rankcondintro}.
Controllers of this form have been used in robust output regulation with full internal models in~\cite{HamPoh00,LogTow97,RebWeiss03}.

In the final part of the paper we illustrate the results
by designing a robustly regulating controller for a laboratory process with five water tanks. 
In the studied experimental setup the restricted class of perturbations arises naturally from considering the unknown valve positions of the water tank system as parameters with uncertainty. The constructed controller containing a reduced order internal model achieves output regulation irregardless of the valve positions.

Robust output regulation with a restricted class of perturbations 
has been studied previously using frequency domain techniques for stable systems in%
~\cite{LaaPau16} by the authors. In this paper we extend the results of~\cite{LaaPau16} most notably 
by generalizing the characterization~\eqref{eq:Robcharintro} of robust controllers \eqref{eq:ControllerIntro} with simple poles to controllers with higher order poles,
by introducing a controller design procedure for unstable plants,
and by establishing the optimality of the presented lower bounds for $\rank (C_k)$.
Locatelli and Schiavoni studied 
a similar control problem
in~\cite{LocSchi11}. However, in~\cite{LocSchi11} 
the controller was required to be robustly regulating in a small neighborhood of a given finite set of plants, and consequently the controller required a full $p$-copy of the exosystem. 

\section{The Robust Output Regulation Problem}

In this section we introduce the notation used in this paper and state the robust output regulation problem. 
We denote the class of functions that are bounded and analytic in the right half plane $\C_+:=\setm{s\in\C}{\re(s)>0}$ by $\Hinf$. 
The set of all matrices of arbitrary size over the set 
$\Hinf$ is denoted by $\mc{M}(\Hinf)$.
We denote the rank, the range, the kernel, and the Moore-Penrose pseudoinverse of a matrix 
$A\in \C^{n\times m}$
by $\rank(A)$, $\ran(A)$, $\ker(A)$, and $A^+$, respectively.

\subsection{Class $\Ops$ of Perturbations}

Throughout the paper we assume that the class $\Ops$ of perturbations has the following properties.
\begin{itemize}
  \item The nominal plant is in the class $\Ops$, i.e., $P(\cdot)\in \Ops$.
  \item Every $\widetilde{P}(\cdot)\in\Ops$ is analytic at the points $\set{i\gw_k}_{k=1}^q$.
\end{itemize}

\subsection{Robust Output Regulation for a Class $\Ops$ of Perturbations}

We consider an error feedback controller of the form
\eqn{
\label{eq:Controller}
C(s)=\sum_{k=1}^q\sum_{n=1}^{q_k}\frac{C^{(k)}_{-n}}{(s-i\gw_k)^n}+C_{0}(s)
}
where $C_{-n}^{(k)}\in \C^{\indim\times p}$, $C_{-q_k}^{(k)}\neq 0$,
$q_k\geq 1$, and 
$C_{0}(\cdot)$ is analytic at $i\gw_k$ for all $k\in \List{q}$.
In particular, the poles of the controller are located at the frequencies $i\gw_k$ of the reference signal~\eqref{eq:yrefintro} and that their orders are
greater than or equal to one.
The plant and the controller form the closed-loop system depicted in Figure~\ref{fig:Closedloop}. Here $\hat{d}$ is an external disturbance. The closed-loop transfer function from $(\yrefhat,\hat{d})$ to $(\hat{e},\hat{u})$ is
\eq{
	H(P,C)=\begin{bmatrix}
		(I-PC)\inv & (I-PC)\inv P\\    	C(I-PC)\inv & I+C(I-PC)\inv P\\
	\end{bmatrix}.
}
\begin{figure}[ht]
	\centering
	\begin{overpic}[scale=0.8]{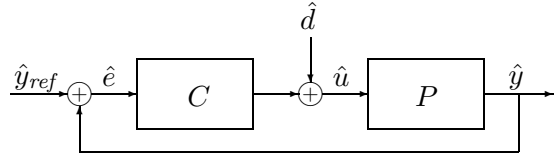}
		\put(1,13){$\yrefhat$}
		\put(17,12.5){$\hat{e}$}
		\put(32.5,9){$C$}
		\put(59,12.5){$\hat{u}$}
		\put(53,23){$\hat{d}$}
		\put(74,9){$P$}
		\put(91,13){$\hat{y}$}
	\end{overpic}
	\caption{The closed-loop system.}
	\label{fig:Closedloop}
\end{figure}

\begin{RORP}
  Given a class $\Ops$ of perturbations, choose the parameters $C_0(\cdot)$ and $C_{-n}^{(k)}$, $k\in\List{q}$ and $n\in\List{q_k}$, of the controller~\eqref{eq:Controller} in such a way that 
  \begin{itemize}
      \setlength{\itemsep}{.7ex}
    \item[\textup{(a)}] The controller $C(\cdot)$ stabilizes the plant $P(\cdot)$, i.e. $H(P,C)\in\mc{M}(\Hinf)$.
    \item[\textup{(b)}] If $\widetilde{P}(\cdot)\in\Ops$ is such that $C(\cdot)$ stabilizes $\widetilde{P}(\cdot)$, then
      \eqn{
      \label{eq:RegCond}
      (I-\widetilde{P}(\cdot)C(\cdot))\inv  \yrefhat(\cdot)\in \mc{M}(\Hinf), 
      }
      where $\yrefhat(\cdot)$ is the Laplace transform of $\yref(t)$ in \eqref{eq:yrefintro}, i.e.,
      \eqn{
      \label{eq:RefSign}
      \yrefhat(s)=\sum_{k=1}^{q}\frac{a_k}{s-i\gw_k}.
      }
  \end{itemize} 
\end{RORP}

If condition~\eqref{eq:RegCond} is satisfied, we say that $C(\cdot)$ \keyterm{regulates} $\widetilde{P}(\cdot)\in\Ops$. In the time-domain, this corresponds to the output $y(\cdot)$ converging to $\yref(\cdot)$ 
asymptotically with respect to time.


%

\section{Characterization of Robustness With Respect to a Class of Perturbations}

In this section we present a characterization for controllers that are robust with respect to a given class $\Ops$ of perturbations.
Since by assumption $\widetilde{P}(\cdot)\in \Ops$ is analytic at $i\gw_k$ and $C(\cdot)$ has pole of order $q_k\geq 1$ at $i\gw_k$, their Laurent series expansions are given by
$$
\widetilde{P}(s)=\sum_{n=0}^{\infty}(s-i\omega_k)^n \widetilde{P}^{(k)}_n, \quad C(s)=\sum_{n=-q_k}^{\infty}(s-i\omega_k)^n C^{(k)}_{n}.
$$
The function $I-\widetilde{P}(\cdot)C(\cdot)$ has the Laurent series expansion 
\eqn{
  \label{eq:IPCLaurent}
  I-\widetilde{P}(s)C(s)=\sum_{n=-q_k}^{\infty}(s-i\gw_k)^n X_n^{(k)}
\shortintertext{where}
\nonumber
  X^{(k)}_n=\gd_{n0}I-\sum_{m=0}^{q_k+n}\widetilde{P}^{(k)}_m C^{(k)}_{n-m}.
}
Here $\gd_{nm}$ is the Kronecker delta.
The following theorem is the main result of this section.

\begin{theorem}\label{thm:Robchar2}
Assume that the controller is of the form \eqref{eq:Controller} and 
let the Laurent series expansion of $I-\widetilde{P}(s)C(s)$ at $i\gw_k$ be given by~\eqref{eq:IPCLaurent}. Then the following hold.
\begin{enumerate}[(i)]
  \item If $\widetilde{P}(\cdot)\in \Ops$ is such that $C(\cdot)$ stabilizes $\widetilde{P}(\cdot)$, then $C(\cdot)$ regulates $\widetilde{P}(\cdot)$ if and only if the equation
      \eqn{\label{eq:Regulationcondition}
      \begin{bmatrix}
      X^{(k)}_{-q_k} & 0 & \cdots & 0\\
      X^{(k)}_{1-q_k} & X^{(k)}_{-q_k} & \cdots & 0\\
      \vdots & \vdots  & \ddots & \vdots\\
      X^{(k)}_{-1} & X^{(k)}_{-2} & \cdots & X^{(k)}_{-q_k}
      \end{bmatrix}
      \begin{bmatrix}
      z^{(k)}_1 \\z^{(k)}_2\\\vdots\\z^{(k)}_{q_k}
      \end{bmatrix}
      =
      \begin{bmatrix}
      0 \\\vdots\\0\\a_k
      \end{bmatrix}
      }
    is solvable for all $k\in\List{q}$.
    For any $k\in\List{q}$ for which $q_k=1$ the solvability of~\eqref{eq:Regulationcondition} is equivalent to
    \eqn{\label{eq:Rangecondition}
    \yrk\in\ran\left(\widetilde{P}(i\gw_k)C^{(k)}_{-1}\right).
    }
    If the plant has the same number of inputs and outputs, i.e., $\indim=p$, and $\widetilde{P}(i\gw_k)$ is invertible,
    then~\eqref{eq:Regulationcondition}
    is equivalent to
    \eq{
    \begin{bmatrix}
    C^{(k)}_{-q_k} & 0 & \cdots & 0\\
    C^{(k)}_{1-q_k} & C^{(k)}_{-q_k} & \cdots & 0\\
    \vdots & \vdots  & \ddots & \vdots\\
    C^{(k)}_{-1} & C^{(k)}_{-2} & \cdots & C^{(k)}_{-q_k}
    \end{bmatrix}
    \begin{bmatrix}
    z^{(k)}_1 \\z^{(k)}_2\\\vdots\\z^{(k)}_{q_k}
    \end{bmatrix}
    =
    \begin{bmatrix}
    0 \\\vdots\\0\\\widetilde{P}^{-1}(i\gw_k)a_k
    \end{bmatrix}
    }
  \item If $C(\cdot)$ stabilizes $P(\cdot)$, then it solves the robust output regulation problem for the class $\Ops$ of perturbations if and only if~\eqref{eq:Regulationcondition} is satisfied for all $\widetilde{P}(\cdot)\in \Ops$ that are stabilized by the controller $C(\cdot)$.
\end{enumerate}
\end{theorem}
\begin{proof}
Part (ii) is a direct consequence of (i) and the statement of the robust output regulation problem. To prove (i), 
let $\widetilde{P}(\cdot)\in \Ops$ be such that $C(\cdot)$ stabilizes $\widetilde{P}(\cdot)$.
Closed-loop stability implies $(I-\widetilde{P}(s)C(s))^{-1}\in\mc{M}(\Hinf)$, and we have the Taylor series
\begin{align*}
(I-\widetilde{P}(s)C(s))^{-1}=\sum_{n=0}^{\infty}(s-i\gw_k)^n H^{(k)}_n.
\end{align*}
at $i\gw_k$. We observe that \eqref{eq:RegCond} is equivalent to the condition
\eqn{
\label{eq:RegCond0}
H_0^{(k)}a_k=0
}
being satisfied for all $k\in\List{q}$. Thus, we need to show that \eqref{eq:RegCond0} and \eqref{eq:Regulationcondition} are equivalent.

Let $k\in\List{q}$ be arbitrary. For simplicity we omit the superscript $(k)$ in $X_n^{(k)}$, $H^{(k)}_n$, and $z_n^{(k)}$. Using the Laurent series expansions and 
$(I-\widetilde{P}(s)C(s))^{-1}(I-\widetilde{P}(s)C(s))=I$, we see that
\begin{align*}
I& =(I-\widetilde{P}(s)C(s))^{-1}(I-\widetilde{P}(s)C(s))\\
& =\sum_{n=-q_k}^{\infty}
\left((s-i\gw_k)^n\sum_{m=0}^{q_k+n} H_m X_{n-m}\right),
\end{align*}
in particular,
\eqn{
\label{eq:main1}
\sum_{m=0}^{q_k+n} H_m X_{n-m} = \begin{cases}
0, & \text{ if } -q_k\leq n<0\\
I, & \text{ if } n=0
\end{cases}.
}
Similarly $(I-\widetilde{P}(s)C(s))(I-\widetilde{P}(s)C(s))^{-1}=I$ implies
\eqn{
\label{eq:main2}
\sum_{m=0}^{q_k+n} X_{n-m} H_m= \begin{cases}
0, & \text{ if } -q_k\leq n<0\\
I, & \text{ if } n=0
\end{cases}.
}

If the condition~\eqref{eq:Regulationcondition} is satisfied, then its last row implies
\eqn{
\label{eq:Regulationcondition3}
H_{0}a_k=\sum_{l=1}^{q_k}H_0 X_{-l}z_l=\sum_{l=0}^{q_k-1}H_0 X_{l-q_k}z_{q_k-l}.
}
Equation~\eqref{eq:main1}
implies
$$H_0X_{-q_k}=0,\quad \text{ and }\quad H_0X_{l-q_k}=-\sum_{m=1}^{l}H_m X_{l-q_k-m}$$
for $l\in\{1,\ldots,q_k-1\}$. Substituting these into~\eqref{eq:Regulationcondition3}, we obtain
\begin{align*}
H_{0}a_k & =-\sum_{l=1}^{q_k-1}\sum_{m=1}^l H_m X_{l-q_k-m}z_{q_k-l}\\
 & = -\sum_{m=1}^{q_k-1}\left(H_m\sum_{l=1}^{q_k-m} X_{-m-l}z_{l}\right).
\end{align*}
This implies \eqref{eq:RegCond0}, since $\sum_{l=1}^{q_k-m} X_{-m-l}z_{l}=0$ for all $m\in\List{q_k-1}$ by~\eqref{eq:Regulationcondition}.

Assume now that~\eqref{eq:RegCond0} holds. Using~\eqref{eq:main2} and~\eqref{eq:RegCond0}
shows that
$$
\sum_{m=1}^{q_k+n}X_{n-m}H_m a_k=\begin{cases}
0, & \text{ if } -q_k+1\leq n<0\\
a_k, & \text{ if } n=0
\end{cases}
$$
which implies that $\displaystyle(z_1,z_2,\ldots,z_{q_k})$ with 
$z_m=H_m a_k$ is a solution of~\eqref{eq:Regulationcondition}.

If $q_k=1$,
then the 
equivalence of~\eqref{eq:Rangecondition} and the solvability of~\eqref{eq:Regulationcondition} follows from  $X_{-1}^{(k)} = -\widetilde{P}(i\gw_k)C_{-1}^{(k)}$.
Finally, denote 
\begin{align*}
T_P  &=
      \begin{bmatrix}
      \widetilde{P}_{0} & 0 & \cdots & 0\\
      \widetilde{P}_{1} & \widetilde{P}_{0} & \cdots & 0\\
      \vdots & \vdots  & \ddots & \vdots\\
      \widetilde{P}_{q_k} & \widetilde{P}_{q_k-1} & \cdots & \widetilde{P}_{0}
      \end{bmatrix}
\shortintertext{and}
T_C &=      \begin{bmatrix}
      C_{-q_k} & 0 & \cdots & 0\\
      C_{1-q_k} & C_{-q_k} & \cdots & 0\\
      \vdots & \vdots  & \ddots & \vdots\\
      C_{-1} & C_{-2} & \cdots & C_{-q_k}
      \end{bmatrix}      
\end{align*}
If 
$\indim=p$
and
$\widetilde{P}_0=\widetilde{P}(i\gw_k)$ is invertible, 
then the final claim follows from 
\begin{align*}
\begin{bmatrix}
      X_{-q_k} & 0 & \cdots & 0\\
      X_{1-q_k} & X_{-q_k} & \cdots & 0\\
      \vdots & \vdots  & \ddots & \vdots\\
      X_{-1} & X_{-2} & \cdots & X_{-q_k}
      \end{bmatrix}
=
T_P T_C.
\end{align*}
\end{proof}

Theorem~\ref{thm:Robchar2} implies the following lower bounds $\sigma_k$ for the order of the internal model. Moreover,
the bounds $\sigma_k$ are optimal in the sense that they can be attained in the construction of controllers.
We call the bound $\sigma_k$ of the following theorem \keyterm{the minimal order of $i\gw_k$ in the internal model.}
Here $\widetilde{P}^{-1}(i\gw_k)\yrk$ denotes the preimage of $\yrk$.

\begin{theorem}\label{thm:MinimumRank}
Let $\Ops$ be a class of perturbations of $P(\cdot)$.
Let $\sigma_k$ be the minimum dimension over all subspaces $\mc{K}_k\subset\C^p$ satisfying $\widetilde{P}^{-1}(i\gw_k)\yrk\cap \mc{K}_k\neq \emptyset$ 
for all $\widetilde{P}(\cdot)\in\Ops$.
\begin{enumerate}[(i)]
\item If $C(\cdot)$ of the form \eqref{eq:Controller}, with $q_k=1$ for some $k\in\List{q}$, stabilizes all the plants in $\Ops$ and solves the robust output regulation problem for $\Ops$, 
then 
\eq{
\mathrm{rank}\left(C_{-1}^{(k)}\right)\geq \sigma_k.
}
\item
Assume that $\indim=p$
and that for some $k\in\List{q}$
$\rank(\widetilde{P}(i\omega_k))=p$ for every $\widetilde{P}(\cdot)\in\Ops$.
If $C(\cdot)$ 
in~\eqref{eq:Controller}
stabilizes all the plants in $\Ops$ and solves the robust output regulation problem for $\Ops$,
then
\eq{
\mathrm{rank}\left(\begin{bmatrix}
C_{-1}^{(k)} & C_{-2}^{(k)} & \cdots & C_{-q_k}^{(k)}
\end{bmatrix}\right)\geq \sigma_k.
}
\item If $r=p$,
$\rank(P(i\omega_k))=p$ for all $k\in\List{q}$, and $P(\cdot)$ is stabilizable, then 
the robust output regulation problem for $\Ops$ can be solved with a
controller~\eqref{eq:Controller}
with $q_k=1$ and
$\mathrm{rank}\bigl(C^{(k)}_{-1}\bigr)= \sigma_k$ for all $k\in\List{q}$.
\end{enumerate}
\end{theorem}
\begin{proof}
Part (iii)
is justified by the construction presented in Section~\ref{sec:ContDesign}.
To prove~(i), let $k\in \List{q}$, $\sigma_0=\rank(C_{-1}^{(k)})$, and let $x_1,\ldots,x_{\sigma_0}$ be linearly independent columns of $C_{-1}^{(k)}$. We set $\mc{K}'=\Span\{x_1,\ldots,x_{\sigma_0}\}$. Since $C(\cdot)$ is robustly regulating, Theorem \ref{thm:Robchar2} implies that for every $\widetilde{P}(\cdot)\in\Ops$ there exists a vector $h$ such that 
$$
a_k=\widetilde{P}(i\gw_k)C^{(k)}_{-1} h=\widetilde{P}(i\gw_k)\sum_{j=1}^{\sigma_0}\alpha_j x_j.
$$
Thus, $\widetilde{P}^{-1}(i\gw_k)a_k\cap \mc{K}'\neq \emptyset$. It follows that $\sigma_0\geq \sigma_k$. Part (ii) follows by similar arguments since 
Theorem~\ref{thm:Robchar2}(i)
implies that there exists 
$h\in \C^{pq_k}$
such that
$$
\begin{bmatrix}
C_{-1}^{(k)} & C_{-2}^{(k)} & \cdots & C_{-q_k}^{(k)}
\end{bmatrix}h=\widetilde{P}^{-1}(i\gw_k)a_k.
$$
\end{proof}

Note that
if $\widetilde{P}(i\gw_k)$ is invertible for all $\widetilde{P}(\cdot)\in\Ops$, then for all $k\in \List{q}$ we have
$\widetilde{P}^{-1}(i\gw_k)\yrk = \widetilde{P}(i\gw_k)^{-1}\yrk$ and
$$\sigma_k=\dim\left(\Span\setm{\widetilde{P}(i\gw_k)^{-1}\yrk}{\widetilde{P}(\cdot)\in\Ops}\right).
$$
Part (iii) of Theorem~\ref{thm:MinimumRank} in particular implies the following.

\begin{corollary}\label{cor:solvability}
The robust regulation problem is solvable if
the plant $P(\cdot)$ is stabilizable, $\indim=p$,
$\rank\left(P(i\omega_k)\right)=p$ for all $k\in\List{q}$, and $\yrk\in\ran(\widetilde{P}(i\gw_k))$ for all $\widetilde{P}(\cdot)\in\Ops$.

\end{corollary}

In \cite{LaaPau16} it was shown that the regulation condition \eqref{eq:Rangecondition} implies 
\eqn{
\label{eq:Robcharlim}
\lim_{s\to i\gw_k} (I-\widetilde{P}(s)C(s))\inv \yrk=0.
}
This means that there exists a transmission zero at $i\omega_k$ blocking the pole of the reference signal. If \eqref{eq:Robcharlim} is satisfied, we say that $(I-\widetilde{P}(s)C(s))\inv$ 
\keyterm{has a transmission zero in the direction $\yrk$}. The aim of the robust regulation is to find a controller that aligns the direction of the transmission zero with the reference signal for every plant in $\Ops$.
An internal model of $i\gw_k$ of order $\sigma_0$ introduces transmission zeros in $\sigma_0$ 
linearly independent directions. In the extreme case $\sigma_0=p$ there is a blocking zero, meaning that the whole transfer function vanishes at $i\omega_k$. This is what is required in classical robust regulation.

\begin{remark}
When considering 
multiple reference signals, the conditions~\eqref{eq:Regulationcondition} are required to be satisfied for each signal.
For example, if we have an additional second reference signal
$$
\yrefhat'=\sum\limits_{k=1}^{q}\frac{b_k}{s-i\omega_k},
$$
then~\eqref{eq:Regulationcondition} for $k\in \List{q}$
are required to be solvable also when $a_k$ is replaced by $b_k$. 
\end{remark}

\section{Controller Design}
\label{sec:ContDesign}

In this section we propose robust controller design method involving two stages. First stage consists of finding a stabilizing controller for the given nominal plant, and in the second stage we construct a robustly regulating controller for the stabilized plant.
We will first show that if the plant is stable, then
we can construct a simple robust controller. 
The design procedure for unstable plants is presented in Section~\ref{sec:Unstable}.

\subsection{Controller Design for a Stable $P(\cdot)$}
\label{sec:Stable}

The following theorem presenting a robust controller for a stable system is the main result of this section.


\begin{theorem}\label{thm:ContDesignStab}
Assume $P(\cdot)$ is stable and $\yrk\in\ran{(\widetilde{P}(i\gw_k))}$ for all $\widetilde{P}(\cdot)\in\Ops$ and $k\in\List{q}$. 
Then the controller
\eqn{\label{eq:ContSimplePole2}
C(s)=\eps \sum_{k=1}^\infty \frac{C_k}{s-i\gw_k}
}
solves the robust output regulation problem for a class $\Ops$ of perturbations if
the design parameters 
$C_k$ and $\epsilon>0$ are chosen in the following way:
\begin{enumerate}
\item Find a subspace $\mc{K}_k$ such that $\widetilde{P}^{-1}(i\gw_k)\yrk\cap \mc{K}_k\neq \emptyset$ and $\mc{K}_k\cap \ker(P(i\omega_k))=\{0\}$.
\item Choose a basis $\{h_1,\ldots,h_{p_k}\}$ of $\mc{K}_k$.
\item Define
    $H_k:=[h_1,\ldots,h_{p_k},0,\ldots,0]$.
\item Choose 
  an invertible
  $D_k\in\C^{p\times p}$ so that the eigenvalues
  of 
\begin{align}\label{eq:Dk}
P(i\gw_k)H_k D_k
\end{align}
are zero or have negative real parts, and that the Jordan blocks related to the zero eigenvalue are trivial.
\item Set
\begin{align}\label{eq:Ck}
C_k:=H_k D_k.
\end{align}
\item Choose sufficiently small $\epsilon>0$ to guarantee closed-loop stability.
\end{enumerate}
\end{theorem}

The proof of Theorem \ref{thm:ContDesignStab} is divided into parts. Lemma~\ref{lem:DesignParameterCk} shows that the proposed controller is regulating and Theorem~\ref{thm:DesignParametere} shows that with the choices made there exists a small enough $\epsilon>0$ guaranteeing closed-loop stability.

Before proceeding further we discuss the choice of $\mc{K}_k$ in the first step of the design procedure. This is the only step that can fail, since such a subspace might not exist. The condition $\widetilde{P}^{-1}(i\gw_k)\yrk\cap \mc{K}_k\neq \emptyset$ is required for the regulation condition, but the stability-related condition $\mc{K}_k\cap \ker(P(i\omega_k))=\{0\}$ is not automatically satisfied if $P(i\omega_k)$ is not injective. This can happen for example if the plant has transmission zeros at $i\omega_k$. 
It is well known that
in order to stabilize the nominal plant with a controller containing a full internal model the plant must not have transmission zeros at the poles of the reference signal \cite{LaaPoh15}. Indeed, $\mc{K}_k\cap \ker(P(i\omega_k))\neq\{0\}$ if $\mc{K}_k=\C^p$ and $\ker\left(P(i\omega)\right)\neq 0$. In our case, $P(\cdot)$ can have transmission zeros 
since $\mc{K}_k$ need not in general
be equal to $\C^p$.

The choice $\mc{K}_k$ achieving the minimal order internal model exists if $P(i\omega_k)$ has full rank and has the same number of inputs and outputs since the condition $\mc{K}_k\cap \ker(P(i\omega_k))=\{0\}$ is then trivially satisfied. 
A particular choice 
would be $\mc{K}_k = \mc{V}_k$ where
\eq{
  \mc{V}_k=\Span\setm{\widetilde{P}^+(i\gw_k)\yrk}{\widetilde{P}(\cdot)\in\Ops}
}
with $p_k:=\dim(\mc{V}_k)$.
In this case we require in addition that
$\mc{V}_k\cap \ker(P(i\omega_k))=\{0\}$, or equivalently, $\dim(P(i\omega_k)\mc{V}_k)=p_k$~\cite{LaaPau16}.
In general, choosing $\mc{K}_k=\mc{V}_k$ is not optimal, since $p_k$ may be strictly greater than the minimal order $\sigma_k$ of the internal model related to the pole $i\gw_k$. However, if 
$\widetilde{P}(i\gw_k)$ 
are invertible for all $\widetilde{P}(\cdot)\in\Ops$, then
item (i) of Theorem \ref{thm:MinimumRank} implies that $\mc{V}_k$ is an optimal choice.

\begin{lemma}\label{lem:DesignParameterCk}
Let $P(\cdot)$ be stable and assume that $\yrk\in\ran({\widetilde{P}(i\gw_k)})$ for all $\widetilde{P}(\cdot)\in\Ops$. If $C_k$ of \eqref{eq:ContSimplePole2} are as in \eqref{eq:Ck}, then the condition \eqref{eq:Rangecondition} holds for every $\widetilde{P}(\cdot)\in\Ops$.
\end{lemma}
\begin{proof}
  Let $\widetilde{P}(\cdot)\in\Ops$ and $k\in \List{q}$ be arbitrary.
Since $\ran(C_k)=\mc{K}_k$ and $\widetilde{P}^{-1}(i\gw_k)\yrk\cap \mc{K}_k\neq \emptyset$ there exists $y$ such that
$
C_k y\in \widetilde{P}^{-1}(i\gw_k) \yrk.
$
It follows that $\yrk=\widetilde{P}(i\gw_k)C_k y$
and thus~\eqref{eq:Rangecondition} holds.
\end{proof}

\begin{lemma}\label{lem:LocalBound}
Let $k\in\List{q}$ be fixed and $P(\cdot)$ be stable.
If $C_k$ of \eqref{eq:ContSimplePole2} are as in \eqref{eq:Ck}, 
then there exists $M\geq 0 $ such that $I-\frac{\eps}{s-i\gw_k} P(i\gw_k)C_k$ are nonsingular and $\norm{(I-\frac{\eps}{s-i\gw_k} P(i\gw_k)C_k)^{-1}}\leq M$
for all $s\in \C_+$ and $\eps>0$.
\end{lemma}

\begin{proof}
By the choice of $C_k$, we know that the Jordan blocks of $P(i\gw_k)C_k$ related to the eigenvalue 0 are trivial, and that the non-zero eigenvalues of $P(i\gw_k)C_k$ have negative real parts. This means that there exist a nonsingular matrix $S$ and a matrix $J$ whose eigenvalues have strictly negative real parts such that
$$
P(i\gw_k)C_k=S\begin{bmatrix}
J & 0\\
0 & 0
\end{bmatrix}
S^{-1},
$$
and for all $z\in \C_+$ we have
$$
\left(I-\frac{1}{z}P(i\gw_k)C_k\right)^{-1}=S\begin{bmatrix}z\left(zI-J\right)^{-1} & 0\\
0 & I
\end{bmatrix}S^{-1}.
$$
Since all the eigenvalues of $J$ have negative real parts, $H(z)=z\left(zI-J\right)^{-1}$ is analytic in $\C_+$. In addition, it approaches $I$ as $|z|\to\infty$. 
Thus, it is uniformly bounded with respect to  $z\in\C_+$. 
It follows that $(I-\frac{\eps}{s-i\gw_k}J)^{-1}$ is uniformly bounded with respect to $s\in\C_+$ and $\eps>0$ since $\setm{\frac{1}{z}}{ z\in \C_+}=\setm{\frac{\eps}{s-i\gw_k}}{ s\in\C_+, \eps>0}$.
\end{proof}

\begin{theorem}\label{thm:DesignParametere}
Let $P(\cdot)$ be stable. 
If $C_k$ of \eqref{eq:ContSimplePole2} are as in \eqref{eq:Ck}, then there exists $\eps^*>0$ such that $C(\cdot)$ of \eqref{eq:ContSimplePole2} stabilizes $P(\cdot)$ for every $\eps\in(0,\eps^*]$.
\end{theorem}

\begin{proof}
  First we show that $(I-P(\cdot)C(\cdot))^{-1}$ is stable for all sufficiently small $\epsilon>0$. To this end, we choose 
  $$
  \gamma<\min\setm{|i\gw_k-i\gw_l|}{1\leq k<l\leq q}
  $$
and define the half discs 
$\Omega_k:=\C_+\cap\setm{s\in\C}{|s-i\gw_k|<\gamma}.$
Our aim is to show the existence of a constant $\eps'>0$ such that $(I-P(\cdot)C(\cdot))^{-1}$ is bounded in $\C_+\setminus \bigcup_{k=1}^q \Omega_k$ whenever $\eps\in(0,\eps']$, and of $\eps_k>0$ such that $(I-P(\cdot)C(\cdot))^{-1}$ is bounded in $\Omega_k$ whenever $\eps\in(0,\eps_k]$. Then $(I-P(\cdot)C(\cdot))^{-1}$ is stable for all $\eps\in(0,\eps^*]$ where $\eps^*=\min\{\eps',\eps_1\ldots,\eps_q\}$.

By the stability of $P(\cdot)$ and the definition of $C(\cdot)$, $P(\cdot)C(\cdot)$ is bounded in $\C_+\setminus \bigcup_{k=1}^q \Omega_k$. Thus, there exists a small enough $\eps'>0$ such that $(I-P(\cdot)C(\cdot))^{-1}$ is bounded in $\C_+\setminus \bigcup_{k=1}^q \Omega_k$ whenever $\eps\in(0,\eps']$.

Next we show the existence of suitable $\eps_k>0$. We write
\eqn{\label{eq:Decomposition}
(I-P(s)C(s))^{-1}&= Q_{1k}(s)(I- \eps Q_{2k}(s)Q_{1k}(s))^{-1}
}
where we have denoted
\eq{ 
& Q_{1k}(s)=\left(I-\frac{\eps{P}(i\gw_k)C_k}{s-i\gw_k}\right)^{-1}, \\
& Q_{2k}(s)=\frac{P(s)-P(i\gw_k)}{s-i\gw_k}C_k-P(s)\sum_{l\neq k} \frac{C_l}{s-i\gw_l}.
}
By Lemma \ref{lem:LocalBound}, $Q_{1k}(s)$ is well-defined 
and uniformly bounded with respect to $s\in \Omega_k$ and $\eps>0$.
In addition, $Q_{2k}(\cdot)$ is bounded in $\Omega_k$ since $P(\cdot)$ and
$\sum_{l\neq k}\frac{C_l}{s-i\gw_l}$
are analytic in $\Omega_k$. 
The decomposition \eqref{eq:Decomposition} implies that we can choose $\eps_k>0$ such that $(I-P(\cdot)C(\cdot))^{-1}$ is bounded in $\Omega_k$ for all $\eps\in(0,\eps_k]$. This completes the proof of the stability of $(I-P(\cdot)C(\cdot))^{-1}$.

Since $P(\cdot)$ is stable, it remains to show the stability of $C(\cdot)(I-P(\cdot)C(\cdot))^{-1}$. By the stability of $(I-P(\cdot)C(\cdot))^{-1}$ and the decomposition \eqref{eq:Decomposition}, we only need to show that
\eq{
  H(s):= \hspace{-.4ex}\frac{\eps}{s-i\gw_k} C_k Q_{1k}(s)
  =C_k \hspace{-.4ex}\left[\frac{s-i\gw_k}{\eps}I- P(i\gw_k)C_k\right]^{-1}
}
does not have pole at $i\gw_k$.
Since $H(s)$ can only have poles of order one, it has the representation
\eq{
H(s)&=C_k\left(\frac{\eps}{s-i\gw_k}E+F_1(s)\right),
}
where $E$ is the projection to $\ker(P(i\gw_k)C_k)$ along $\ran(P(i\gw_k)C_k)$ and $F_1(s)$ is an analytic function \cite{Roth81}. Since $\mc{K}_k\cap \ker(P(i\omega_k))=\{0\}$ and $\ran(C_k)=\mc{K}_k$ we have that $\ker(P(i\gw_k)C_k)=\ker(C_k)$. Consequently, $C_k E=0$, and $H(s)$ is analytic at $i\gw_k$. This completes the proof. 
\end{proof}

\subsection{Controller Design for an Unstable $P(\cdot)$}
\label{sec:Unstable}

For unstable plants, the design procedure is given in the following theorem. 
It is based on the two stage approach proposed in \cite[Section 5.3]{Vid85}.

\begin{theorem}\label{thm:ContDesignUnstab}
If the steps of items 1 and 2 below can be carried out, then the controller of step 3 is robustly regulating.
\begin{enumerate}
\item 
Stabilize the nominal plant $P(\cdot)$ using a controller $C_s(\cdot)$ that does not have poles at $i\gw_k$ for $k\in \List{q}$.
\item Find a controller $C_r(\cdot)$ of form \eqref{eq:Controller} that stabilizes 
  \eq{
    P_s(\cdot):=P(\cdot)\left(I-C_s(\cdot)P(\cdot)\right)^{-1} 
  }
and satisfies the condition \eqref{eq:Regulationcondition} for every $\widetilde{P}(\cdot)\in\Ops$.

\item A robustly regulating controller of $P(\cdot)$ is given by
  \eq{
    C(\cdot)=C_s(\cdot)+C_r(\cdot).
  }
\end{enumerate}
\end{theorem}

\begin{remark}\label{rem:ContDesign}
Step 2 can be completed using the approach for stable plants in Section~\ref{sec:Stable} by choosing
the matrices $H_k$ associated to $\Ops$ as before, but 
replacing $P(\cdot)$ by $P_s(\cdot)$ when choosing $D_k$ in~\eqref{eq:Dk}. 
In particular, if the plant $P(\cdot)$ is invertible at $i\omega_k$, then so is the stabilized plant $P_s(\cdot)$, and
thus it is 
possible to carry out Step 2.
\end{remark}


\textit{Proof of Theorem~\ref{thm:ContDesignUnstab}.}
Theorem 5.3.6 of \cite{Vid85} (the generalization to the current case follows by \cite[Section 8]{Vid85}) shows that $C(\cdot)$ stabilizes $P(\cdot)$ since $C_s(\cdot)$ stabilizes $P(\cdot)$ and $C_r(\cdot)$ stabilizes $P_s(\cdot)$. Thus we only need to show that $C(\cdot)$ is regulating for $\widetilde{P}(\cdot)\in\Ops$. Since $C_s(\cdot)$ does not possess poles at $i\gw_k$ and $C_r(\cdot)$ is of the form \eqref{eq:Controller}, it is obvious that $C(\cdot)$ is of the form \eqref{eq:Controller} as well with the same 
matrices $C_{-n}^{(k)}$. The matrices $C_{-n}^{(k)}$ satisfy the condition \eqref{eq:Regulationcondition} by assumption.
\hfill $\square$ 

\section{Example}
\label{sec:example}

Let us consider the laboratory process of Figure~\ref{fig:tanks} with five water tanks. There is an opening in the bottom of each water tank and the water from the tanks four and five flows to the tanks below them. The three pumps with operating voltages $u_j$, $j=1,2,3$, induce a flow $r_j u_j$ where $r_j$ is a constant. 
The outputs $y_l$, $l=1,2,3$, of the plant are defined as the deviations 
from the initial water levels in 
Tanks $l$, $l=1,2,3$, respectively. The initial water level, as well as other properties of the system, are chosen so that no complications such as negative water levels can occur.
The aim of our control problem is to choose the inputs $u_j$ so that the output $y(t)=(y_1(t),y_2(t),y_3(t))$ converges to the reference signal
$$
y_{ref}(t)=(\sin(t),1,1),
$$
i.e. the water level of Tank 1 is changing in a periodic manner while kept one unit above the initial level in the other two bottom tanks.

\begin{figure}[h]
\centering
	\setlength{\unitlength}{.25cm}
	\begin{picture}(31, 23)
	
	\thicklines

	\put(2,1.5){\fcolorbox{black!15}{black!15}{\begin{minipage}[t][0.25cm]{7.25cm}\makebox[1cm]{ }\end{minipage}}}
	\put(2,3){\line(0,-1){3}}
	\put(2,0){\line(1,0){30}}
	\put(32,0){\line(0,1){3}}

	\put(5,7.5){\fcolorbox{black!15}{black!15}{\begin{minipage}[t][0.5cm]{1cm}\makebox[1cm]{ }\end{minipage}}}
	\put(5,5){\line(0,1){6}}
	\put(10,5){\line(0,1){6}}
	\put(5,5){\line(1,0){2}}
	\put(8,5){\line(1,0){2}}

	\put(15,8.5){\fcolorbox{black!15}{black!15}{\begin{minipage}[t][0.75cm]{1cm}\makebox[1cm]{ }\end{minipage}}}
	\put(15,5){\line(0,1){6}}
	\put(20,5){\line(0,1){6}}
	\put(15,5){\line(1,0){2}}
	\put(18,5){\line(1,0){2}}

	\put(25,7.5){\fcolorbox{black!15}{black!15}{\begin{minipage}[t][0.5cm]{1cm}\makebox[1cm]{ }\end{minipage}}}
	\put(25,5){\line(0,1){6}}
	\put(30,5){\line(0,1){6}}
	\put(25,5){\line(1,0){2}}
	\put(28,5){\line(1,0){2}}

	\put(5,15.5){\fcolorbox{black!15}{black!15}{\begin{minipage}[t][0.25cm]{1cm}\makebox[1cm]{ }\end{minipage}}}
	\put(5,14){\line(0,1){6}}
	\put(10,14){\line(0,1){6}}
	\put(5,14){\line(1,0){2}}
	\put(8,14){\line(1,0){2}}

	\put(15,17.5){\fcolorbox{black!15}{black!15}{\begin{minipage}[t][0.75cm]{1cm}\makebox[1cm]{ }\end{minipage}}}
	\put(15,14){\line(0,1){6}}
	\put(20,14){\line(0,1){6}}
	\put(15,14){\line(1,0){2}}
	\put(18,14){\line(1,0){2}}
	

	\thinlines

	\put(2,2){\line(1,0){30}}
	\put(5,8){\line(1,0){5}}
	\put(15,9){\line(1,0){5}}
	\put(25,8){\line(1,0){5}}
	\put(5,16){\line(1,0){5}}
	\put(15,18){\line(1,0){5}}

	\put(6.5,9){\makebox(2,2){$1$}}
	\put(16.5,9){\makebox(2,2){$2$}}
	\put(26.5,9){\makebox(2,2){$3$}}
	\put(6.5,18){\makebox(2,2){$4$}}
	\put(16.5,18){\makebox(2,2){$5$}}
		
	\small

	\put(3,1){\line(0,1){3}}
	\put(3,6){\line(0,1){5.5}}
	\put(3.5,12){\line(1,0){3}}
	\put(6.5,12){\line(0,-1){6}}
	\put(3,12.5){\line(0,1){9.5}}
	\put(3,22){\line(1,0){13.5}}
	\put(16.5,22){\line(0,-1){7}}

	\put(13,1){\line(0,1){3}}
	\put(13,6){\line(0,1){5.5}}
	\put(13.5,12){\line(1,0){3}}
	\put(16.5,12){\line(0,-1){6}}
	\put(13,12.5){\line(0,1){8.5}}
	\put(13,21){\line(-1,0){4.5}}
	\put(8.5,21){\line(0,-1){6}}

	\put(23,1){\line(0,1){3}}
	\put(23,6){\line(0,1){5.5}}
	\put(23.5,12){\line(1,0){3}}
	\put(26.5,12){\line(0,-1){6}}
	\put(23,12.5){\line(0,1){8.5}}
	\put(23,21){\line(-1,0){4.5}}
	\put(18.5,21){\line(0,-1){6}}

	\put(3,5){\circle{2}}
	\put(2,4){\makebox(2,2){$u_1$}}

	\put(3,12){\circle{1.1}}
	\put(0,11){\makebox(2,2){$\gamma_1$}}

	\put(13,5){\circle{2}}
	\put(12,4){\makebox(2,2){$u_2$}}
	
	\put(13,12){\circle{1.1}}
	\put(10,11){\makebox(2,2){$\gamma_2$}}

	\put(23,5){\circle{2}}
	\put(22,4){\makebox(2,2){$u_3$}}
	
	\put(23,12){\circle{1.1}}
	\put(20,11){\makebox(2,2){$\gamma_3$}}
	
	\end{picture}
\caption{A five tank laboratory process.}
\label{fig:tanks}
\end{figure}
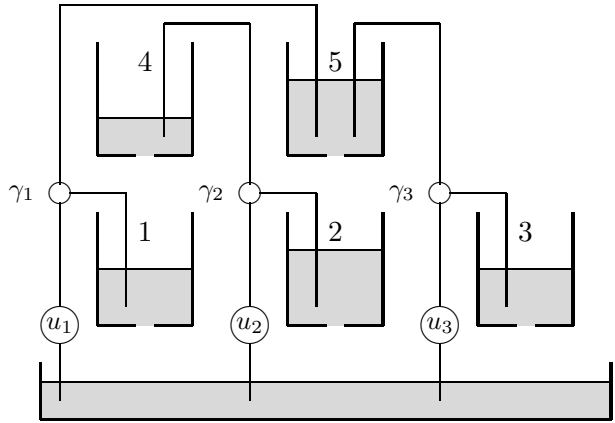

The parameters $0<\gamma_k< 1$ correspond to how the three valves are set prior to the experiment. The flow induced by the first pump to Tank 1 is $\gamma_1 r_1 u_1$ and $(1-\gamma_1) r_1 u_1$ to Tank 5, and similarly for the other two valves. The changes to the valve positions can be considered as perturbations to the system.

The transfer function of the system linearized at the initial water levels is
$$
\widetilde{P}(s)=\begin{bmatrix}
 \frac{\gamma_1 \alpha_{1}}{s+\beta_1} & \frac{(1-\gamma_{2}) \alpha_{2}}{(s+\beta_1)(s+\beta_2)} & 0\\
\frac{(1-\gamma_{1}) \alpha_{3}}{(s+\beta_3)(s+\beta_4)} &  \frac{\gamma_2 \alpha_{4}}{s+\beta_3} & \frac{(1-\gamma_{3}) \alpha_{5}}{(s+\beta_3)(s+\beta_4)}\\
0 & 0 &  \frac{\gamma_3 \alpha_{6}}{s+\beta_5}
\end{bmatrix}
$$
where the parameters $\alpha_j$ and $\beta_l$ depend on the tank cross-sections, the outlet hole cross-sections, constants of proportionality $k_l$, and the initial water levels. For more details, see \cite{Joh00} where a similar system with four tanks was considered.

Here we choose the initial setup so that $\alpha_1=\alpha_2=\alpha_3=\beta_1=\beta_2=\beta_3=1$ and $\alpha_4=\alpha_5=\alpha_6=\beta_4=\beta_5=2$ for simplicity, i.e.
we have
$$
\widetilde{P}(s)=\begin{bmatrix}
 \frac{\gamma_1}{s+1} & \frac{1-\gamma_{2}}{(s+1)^2} & 0\\
\frac{1-\gamma_{1} }{(s+1)(s+2)} &  \frac{2\gamma_2}{s+1} & \frac{2(1-\gamma_{3})}{(s+1)(s+2)}\\
0 & 0 &  \frac{2\gamma_3}{s+2}
\end{bmatrix}
$$
Let the initial positions of the valves be $\gamma_1=\gamma_2=\gamma_3=\frac{1}{2}$, i.e. the nominal plant is
$$
P(s)=\frac{1}{2}\begin{bmatrix}
 \frac{1}{s+1} & \frac{1}{(s+1)^2} & 0\\
\frac{1}{(s+1)(s+2)} &  \frac{2}{s+1} & \frac{2}{(s+1)(s+2)}\\
0 & 0 &  \frac{2}{s+2}
\end{bmatrix}
$$
The frequency domain representation of $y_{ref}(t)$ is
\begin{align*}
\yrefhat(s)& =\frac{1}{s-i}\begin{bmatrix}
\frac{-i}{2}\\0\\0
\end{bmatrix}
+
\frac{1}{s+i}\begin{bmatrix}
\frac{i}{2}\\0\\0
\end{bmatrix}
+
\frac{1}{s}\begin{bmatrix}
0\\1\\1
\end{bmatrix}\\
& =\frac{1}{s-i}a_1
+
\frac{1}{s+i}a_{-1}
+
\frac{1}{s}a_0.
\end{align*}

In order to
solve the robust regulation problem,
we define
\eq{
\mathcal{V}_k &=\Span\setm{\widetilde{P}^+(i\gw_k)\yrk}{\gamma_1,\, \gamma_2,\, \gamma_3\,\in (0,1)}
}
for $\gw_{k}=k$ and $k=-1,0,1$. Let $\vec{e}_l$ be the $l$th natural basis vector of $\C^3$. Because of the upper triangular block structure of $\widetilde{P}(i\gw_k)$, it is easy to deduce that
$\mathcal{V}_{-1}=\mathcal{V}_1=\Span\{\vec{e}_1,\vec{e}_2\}$
and $\mathcal{V}_0=\mathbb{C}^3$. Following the design procedure of Theorem~\ref{thm:ContDesignStab}, we choose $H_{-1}=\mathrm{diag}(1,1,0)=H_1$ and $H_0=I$. Now the controller \eqref{eq:ContSimplePole2} satisfies the regulation property \eqref{eq:Rangecondition}.

It remains to choose invertible $D_k$ and small enough $\eps>0$ to guarantee stability. We can choose $D_k=-I$ since all the eigenvalues of $P(ik)$ have positive real parts for every $k=-1,0,1$. If we choose $\eps=1$, then we have
\begin{align*}
C(s) & =-\left(\frac{1}{s+i}+\frac{1}{s-i}\right)\begin{bmatrix}
1 & 0 & 0\\
0 & 1 & 0\\
0 & 0 & 0
\end{bmatrix}
-\frac{1}{s}\begin{bmatrix}
1 & 0 & 0\\
0 & 1 & 0\\
0 & 0 & 1
\end{bmatrix}\\
 & =\begin{bmatrix}
-\frac{3s^2+1}{s^3+s} & 0 & 0\\
0 & -\frac{3s^2+1}{s^3+s} & 0\\
0 & 0 & -\frac{1}{s}
\end{bmatrix}.
\end{align*}
In order to show that the controller is stabilizing for $P(\cdot)$ we note that the proof of Theorem \ref{thm:DesignParametere} shows that it is sufficient that $(I-P(\cdot)C(\cdot))^{-1}$ is stable. The stability follows now by observing that the zeros of
$$
\det(I-P(s)C(s))=\frac{\left\{ \begin{array}{c}
4s^{10} + 20s^9 + 62s^8 + 140s^7\\ + 216s^6 + 262s^5 + 217s^4\\ + 136s^3 + 58s^2 + 18s + 3
\end{array}
\right\}}{4s^3(s + 1)(s + 2)^2(s^2 + 1)^2}
$$
have negative real parts, i.e. $(I-P(\cdot)C(\cdot))^{-1}$ cannot have poles in the closed right half plane $\C_+$.

\begin{figure}
\includegraphics[scale=.6]{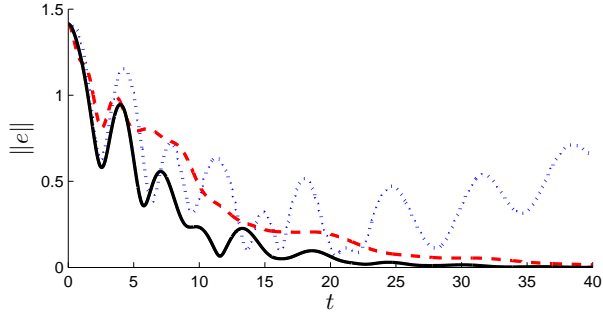}
\caption{The output performance of the closed loop system}
\label{fig:performance}
\end{figure}

The achieved output performance is illustrated in Fig. \ref{fig:performance} for the nominal plant (solid line) and for two perturbed plants where the error norm $\|e(t)\|$ 
is plotted with respect to $t$. Convergence to the reference signal is guaranteed only for stabilized plants. The first perturbed plant (dashed line) with valve positions $\gamma_1=0,7$, $\gamma_2=0,9$, and $\gamma_3=0,2$ is stabilized by $C(\cdot)$, so the convergence follows. However, instability of the closed loop system with the second perturbed plant (dotted line) having valve positions $\gamma_1=0,25$, $\gamma_2=0,25$, and $\gamma_3=0,45$, causes undesired output behavior.

We end this section by comparing the proposed design procedure with the classical one. Here we have used the knowledge that the first two inputs do not affect the third output, i.e., we have structured perturbations, whereas the classical design procedure ignores this fact and the perturbations are taken to be totally arbitrary. 
In our controller the internal model is minimal in the sense that the ranks of the matrices $C_k$ for $k=-1,0,1$ are minimal. Since $C_1$ and $C_{-1}$ have rank two instead of rank three, which would be the case if the classical approach is used, the order of the controller's realization is reduced by two. Finally, the possible numerical inaccuracies when determining $\ran(\widetilde{P}(i\gw_k)C_k)$ can result unwanted behavior in general, which does not happen in the classical approach as long as the closed-loop system remains stable. However, small errors for $\ran(\widetilde{P}(i\gw_k)C_k)$ result only to small errors in regulation. More importantly, no such issues arise in our example
since we can determine the structure of the system without using numerical estimations.

\section{Conclusions}

We have studied robust regulation 
in the situation where the class of perturbations is restricted.
As our main result we presented necessary and sufficient conditions for a stabilizing controller to be robust 
with respect to a given class of perturbations.
Our results in particular show that depending on the class of perturbations
the size of the internal model in the controller can in some situations be reduced.
We introduced a design procedure for constructing a robustly regulating controller with a minimal internal model.
In this paper we have considered reference signals that are trigonometric polynomials, and future research topics include extending the results for more general reference signals including polynomially increasing functions.

\bibliographystyle{plain}

\end{document}